\documentclass[12pt]{amsart}
\usepackage{amsmath,amsfonts,amsthm,amssymb,xypic}

\newcommand{\scrC}{\mathcal{C}} 
\newcommand{\scrF}{\mathcal{F}}
\newcommand{\scrG}{\mathcal{G}}
\newcommand{\scrI}{\mathcal{I}}
\newcommand{\scrM}{\mathcal{M}}
\newcommand{\scrO}{\mathcal{O}}
\newcommand{\scrQ}{\mathcal{Q}}

\newcommand{\bA}{\mathbf{A}}
\newcommand{\bL}{\mathbf{L}}
\newcommand{\bR}{\mathbf{R}}

\newcommand{\bbD}{\mathbb{D}}
\newcommand{\bbF}{\mathbb{F}}
\newcommand{\bbG}{\mathbb{G}}
\newcommand{\bbZ}{\mathbb{Z}}

\newcommand{\supp}{\mathrm{supp}}
\newcommand{\st}{\mathrm{star}}
\newcommand{\Pic}{\mathrm{Pic}}
\newcommand{\alt}{\mathrm{alt}}
\newcommand{\step}{\mathrm{step}}
\newcommand{\Spec}{\mathrm{Spec}}

\newcommand{\Hom}{\mathrm{Hom}}

\newcommand{\spA}{{^\bA\!}}
\newcommand{\spAp}{{^{\bA,p}\!}}
\newcommand{\uchi}{\underline{\chi}}
\newcommand{\cod}{\mathrm{cod}}

\newtheorem{dummy}{Dummy}[section]

\newtheorem{proposition}[dummy]{Proposition}
\newtheorem{corollary}[dummy]{Corollary}
\newtheorem{theorem}[dummy]{Theorem}
\newtheorem*{uncorollary}{Corollary}
\newtheorem*{maintheorem}{Main Theorem}
\theoremstyle{definition}
\newtheorem{definition}[dummy]{Definition}
\newtheorem{remark}[dummy]{Remark}
\newtheorem{remarks}[dummy]{Remarks}
\newtheorem{warning}[dummy]{Warning}

\title{Staggered t-structures on toric varieties}
\author{David Treumann}
\date{June 3, 2008}

\begin{document}

\maketitle

\begin{abstract}
Achar has recently introduced a family of t-structures on the derived category of equivariant coherent sheaves on a $G$-scheme, generalizing the perverse coherent t-structures of Bezrukavnikov and Deligne.  They are called \emph{staggered} t-structures, and their main point of interest so far is that they are more often self-dual.  In this paper we investigate these t-structures on the $T$-equivariant derived category of a toric variety.
\end{abstract}

\section{Introduction}

Let $G$ be an algebraic group, let $Y$ be a $G$-scheme, and let $D^b_G(Y)$ denote the bounded derived category of $G$-equivariant coherent sheaves on $Y$.  Achar has introduced \cite{pramod} a family of t-structures -- the \emph{staggered} t-structures -- on $D^b_G(Y)$ generalizing the perverse coherent t-structures of Bezrukavnikov and Deligne \cite{bezru}.  

Serre-Grothendieck duality provides an equivalence  $D^b_G(Y)^{op} \cong D^b_G(Y)$.  So far, the main point of interest about the staggered t-structures is that they are more often self-dual (up to shifts and line bundles, or put another way with respect to right choice of dualizing sheaf) than the Bezrukavnikov-Deligne t-structures.  Indeed, for the Bezrukavnikov-Deligne t-structures to be self-dual it is necessary that the dimension of the $G$-orbits are either all even or all odd, which rules out many interesting group actions.

The purpose of this paper is to parameterize the staggered t-structures combinatorially when $G$ is a torus and $Y$ is a toric variety.  The following, which is proved ultimately in theorem \ref{thmpervtor}, is our main result:

\begin{maintheorem}

Let $T$ be a torus and let $Y$ be a $T$-toric variety with associated fan $\Sigma$ in $X_*(T)$.  For each cone $C$ of $\Sigma$, let $O_C \subset Y$ denote the reduced $T$-orbit corresponding to $C$, let $i_C:O_C \hookrightarrow Y$ denote the inclusion map, and let $\bL i_C^*:D^b_T(Y) \to D^-_T(O_C)$ and $\bR i_C^! :D^b_T(Y) \to D^+_T(O_C)$ denote the derived restriction and corestriction functors.

Let $\bA$ be an ``s-structure'' on $\Sigma$ (definition \ref{sfan}) -- that is, a collection of elements $A_C \in -C$, for each cone $C \in \Sigma$, with the further property that whenever $A_C = 0$ and $C' \subset C$ is a face of $C$, we have $A_{C'} = 0$ as well.  

Let $p$ be a $\bbZ$-valued function on the set of cones in $\Sigma$, and suppose that $p$ satisfies the perversity conditions of definition \ref{pfan}.  Define full subcategories
$\spAp D^{\leq 0}$ and $\spAp D^{\geq 0}$ of $D^b_T(Y)$ by
$$\begin{array}{ccc}
\spAp D^{\leq 0} & =  & \{ \scrF \mid \forall C\, \langle \xi, A_C\rangle \leq p(C) - k \text{ \rm whenever }\xi\text{ \rm is a weight
of }h^k(\bL i_C^* \scrF)\} \\
\spAp D^{\geq 0} & = & \{ \scrF \mid \forall C \, \langle \xi, A_C \rangle \geq p(C) - k \text{ \rm whenever } \xi \text{ \rm is a weight of }h^k(\bR i_C^! \scrF)\}
\end{array} 
$$
Then $(\spAp D^{\leq 0}, \spAp D^{\geq 0})$ is a bounded, nondegenerate t-structure on $D^b_T(Y)$.  We will call it the \emph{$(\bA,p)$-staggered t-structure}.  The $(\bA,p)$-staggered t-structures, as $\bA$ runs over s-structures and $p$ runs over perversity functions, are precisely the t-structures constructed in \cite{pramod}.
\end{maintheorem}

We remark that the definition of s-structure given in \cite{pramod} is fairly complicated and inexplicit.  The fact that, at least on toric varieties, these s-structures are equivalent to the more concrete objects defined in the main theorem is proved in theorem \ref{thm-stor} and may also be of interest.

It is easy to combine the main theorem with Achar's general results to exhibit many self-dual examples of staggered t-structures on toric varieties.  
Recall that we may identify the equivariant Picard group $\Pic_T(Y)$ of a toric variety with the group of $\bbZ$-valued piecewise-linear functions on the fan $\Sigma$ (see section \ref{sec-alttor}).  If $\bA$ is an s-structure on $\Sigma$, we have for each $C$ an evaluation map $\uchi \mapsto \uchi(-A_C)$, where $\uchi$ denotes the piecewise-linear function on $\Sigma$ corresponding to an element of $\Pic_T(Y)$.

\begin{uncorollary}
Let $T \cong \bbG_m^n$ be a torus, let $Y$ be a $T$-toric variety, and let $\bA$ be an s-structure on the fan $\Sigma$ associated to $Y$.  Then the following are equivalent:
\begin{enumerate}
\item There exists a perversity function $p$ and a dualizing complex $\bbD$ with the property that the $(\bA,p)$-staggered t-structure is self-dual with respect $\bbD$.
\item There exists a piecewise-linear function $\uchi$ on $\Sigma$ such that $\mathrm{dim}(C)+ \uchi(-A_C)$ is even for all $C$, and such that 
$$\uchi(-A_{C'}) - \uchi(-A_{C}) +1 \geq 0$$
for all pairs of cones $C,C'$ with $C$ a codimension one face of $C'$.
\end{enumerate}
\end{uncorollary}

A more precise version of this is stated in corollary \ref{cor-dual}.

\subsection{Notation and conventions}
We will work throughout over a fixed a field $\bbF$.  If $G$ is an affine algebraic group and $Y$ is a
noetherian $G$-scheme, let $\scrC_G(Y)$ and $\scrQ_G(Y)$ denote the abelian categories of $G$-equivariant coherent sheaves and $G$-equivariant quasicoherent sheaves on $Y$.  We will work with the following triangulated categories:
\begin{enumerate}
\item $D^b_G(Y) := D^b(\scrC_G(Y))$ is the bounded derived category of $\scrC_G(Y)$.
\item $D^-_G(Y) := D^-(\scrC_G(Y))$ is the bounded-above derived category of $\scrC_G(Y)$.
\item $D^+_G(Y)$ is the full subcategory of $D^+(\scrQ_G(Y))$ whose objects are complexes whose cohomology belongs to $\scrC_G(Y)$.
\end{enumerate}
We will denote the $i$th cohomology sheaf of an object $\scrF$ of one of these categories by $h^i(\scrF) \in \scrC_T(Y)$.  We will use $H^i(Y;\scrF)$ to denote the usual sheaf cohomology groups.

When $f:X \hookrightarrow Y$ is a $G$-equivariant inclusion, we have pullback operations on these derived categories which we will denote by $\bL f^*$ and $\bR f^!$.  The functor $\bR f^!:D^+_G(Y) \to D^+_G(X)$ is the right derived functor of the functor $f^!:\scrQ_G(Y) \to \scrQ_G(X)$ that takes  a sheaf to its maximal subsheaf scheme-theoretically supported on $X$.  More care has to be taken to define $\bL f^*$: it is not known whether there are always enough $G$-equivariant vector bundles to define $\bL f^*$ as a left-derived functor, even on toric varieties.  In this paper we will only make use of $\bL f^*$ (and $\bR f^!$) when $f$ may be written as a composition $j \circ i$, where $j$ is the inclusion of a $G$-stable affine open subscheme and $i$ is a closed embedding. On an affine scheme we do have enough equivariant vector bundles, and we may define $\bL f^*:D^-_G(Y) \to D^-_G(X)$ to be the composition of $j^*$ and the left derived functor of $i^*$.

We will discuss our conventions for toric varieties in section \ref{sec-torvar}.

\section{Review of $\mathrm{s}$-structures}
\label{sectwo}

The staggered t-structures are associated to pairs $(\bA,p)$ where $\bA$ is an ``s-structure'' and $p$ is a ``perversity function.''  The role played by $p$ is analogous to that of the original perversity functions in \cite{bbd} (and to the perversity functions appearing in the coherent setting in \cite{bezru}), but the notion of an ``s-structure'' is an innovation of \cite{pramod} whose place in the construction of t-structures is new.  The definition of s-structure is somewhat long, and its geometric meaning is still unclear at least to me, but in the next section we will see that s-structures on toric varieties are in correspondence with fairly simple combinatorial data.

In this section, let $G$ be an affine algebraic group, let $Y$ be a $G$-scheme, and let $\scrC_G(Y)$ denote the abelian category of $G$-equivariant coherent sheaves on $Y$.  Let us furthermore assume that $G$ acts on $Y$ with finitely many orbits.   This assumption applies to toric varieties and, in light of
 \cite[Theorem 10.2]{pramod} (reproduced as theorem \ref{achargluing} below), it allows us to simplify the definition of an s-structure.

\begin{definition}
Let $\scrC$ be an abelian category.  A \emph{Serre filtration} of $\scrC$ is a collection $\{\scrC_{\leq w}\}_{w \in \bbZ}$ of Serre subcategories of $\scrC$ with $\scrC_{\leq w} \subset \scrC_{\leq w+1}$ for each $w$.  We will sometimes denote a Serre filtration of an abelian category $\scrC$ by a symbol $\bA$.  Then we will denote the filtered pieces of $\scrC$ by ${^\bA \!} \scrC_{\leq w}$.
\end{definition}

An s-structure on a category of equivariant coherent sheaves is completely determined by a Serre filtration.  Of course, for a Serre filtration to determine an s-structure it must satisfy some conditions.  Before describing these we need some more notation.

\begin{definition}
Let $\scrC$ be an abelian category and let $\bA$ be a Serre filtration of $\scrC$.  
\begin{enumerate}
\item Define a new, decreasing filtration of $\scrC$ by full subcategories ${^\bA\!}\scrC_{\geq w}$, by setting
$${^\bA \!}\scrC_{\geq w} = \{c \mid \Hom(x,c) =  0 \text{ for all } x \in {^\bA \!}\scrC_{\leq w-1}\}$$
\item Suppose $\scrC = \scrC_G(Y)$.  Let $Y_0 \subset Y$ be the union of open $G$-orbits.  Then define a second decreasing filtration of $\scrC$ by setting
$${^\bA\!}\tilde{\scrC}_{\geq w} = \{\scrF \mid \text{there exists an $\scrF_1 \in {^\bA\!}\scrC_{\geq w}$ such that $\scrF_1 \vert_{Y_0} \cong \scrF\vert_{Y_0}$}\}$$
\end{enumerate}
\end{definition}

Thus, $\spA\scrC_{\geq w}$ is the ``right orthogonal'' to $\spA\scrC_{\leq w-1}$.  Note that the $\spA\scrC_{\geq w}$ are not usually Serre subcategories, though they are closed under subobjects and extensions.

Now we are ready to introduce s-structures:

\begin{definition}
\label{def-s-struct}
Let $G$ be an affine algebraic group, let $Y$ be a scheme on which $G$ acts with finitely many orbits, and let $\scrC_G(Y)$ be the category of $G$-equivariant coherent sheaves on $Y$.  A Serre filtration $\bA$ of $\scrC_G(Y)$ is an \emph{s-structure} if it satisfies the following conditions:

\begin{enumerate}
\item[(S4)] For any $\scrF \in \scrC_G(Y)$ and any $w \in \bbZ$, there is a short exact sequence
$$0 \to \scrF' \to \scrF \to \scrF'' \to 0$$
with $\scrF' \in \spA \scrC_{\leq w}$ and $\scrF'' \in \spA \scrC_{\geq w+1}$
\item[(S5)] For any $\scrF \in \scrC_G(Y)$, there are integers $w$ and $v$ such that $\scrF \in \spA\scrC_{\geq w}$ and $\scrF \in \spA\scrC_{\leq v}$.
\item[(S6)] If $\scrF \in \spA\scrC_{\leq w}$ and $\scrG \in \spA\scrC_{\leq v}$, then $\scrF \otimes \scrG \in \spA\scrC_{\leq w+v}$.
\item[(S7)] For each $w \in \bbZ$, $\spA\tilde{\scrC}_{\geq w}$ is a Serre subcategory of $\scrC$.
\item[(S8)] If $\scrF \in \spA\tilde{\scrC}_{\geq w}$ and $G \in \spA\tilde{\scrC}_{\geq v}$, then $\scrF \otimes \scrG \in \spA\tilde{\scrC}_{\geq w+v}$.
\end{enumerate}
\end{definition}

\begin{remark}
Technically this is Achar's definition of an ``almost s-structure''.  We have omitted axioms (S9) and (S10).  (We have also omitted axioms (S1)-(S3).  They define what we have called a Serre filtration.)  However, as is pointed out in \cite{pramod}, it is a consequence of theorem \ref{achargluing} that every almost s-structure is an s-structure when $G$ acts with finitely many orbits.
\end{remark}

An s-structure endows every equivariant coherent sheaf with a canonical filtration.  Indeed, axiom (S4) provides a functor $\sigma_{\leq w}:\scrC_G(Y) \to \scrC_G(Y)$ called \emph{s-truncation}, which carries a sheaf $\scrF$ to its maximal subsheaf contained in $\spA \scrC_G(Y)_{\leq w}$.  The functor $\sigma_{\leq w}$ is left exact, since we have a natural inclusion $\sigma_{\leq w} \scrF \subset \scrF$.

\subsection{Induced and glued s-structures}

Given an s-structure on a $G$-scheme $Y$, we may define an s-structure on any locally closed $G$-stable subscheme of $Y$:

\begin{definition}
Let $G$ be an affine algebraic group and let $Y$ be a scheme on which $G$ acts with finitely many orbits.  Let $\bA$ be a Serre filtration of $\scrC_G(Y)$.
\begin{enumerate}
\item For each $G$-stable open subscheme $U$ of $Y$, define a Serre filtration on $\scrC_G(U)$ by
$$\spA\scrC_G(U)_{\leq w} = 
\{ \scrF \mid 
\text{
there exists an $\scrF_1 \in \scrC_G(X)_{\leq w}$ such that $\scrF_1\vert_U \cong \scrF$} \}$$

\item For each $G$-stable closed subscheme  $i:Z \hookrightarrow Y$, define a Serre filtration on $\scrC_G(Z)$ by

$$\spA\scrC_G(Z)_{\leq w} = \{ \scrF \mid i_* \scrF \in \scrC_G(Y)_{\leq w}\}$$
\end{enumerate}
\end{definition}

It is proved in \cite[Lemma 3.9]{pramod} that $\{\spA\scrC_G(U)_{\leq w}\}$ and $\{\spA\scrC_G(Z)_{\leq w}\}$ are s-structures whenever $\{\spA\scrC_G(X)_{\leq w}\}$ is an s-structure.  The following gluing theorem is a sort of converse:

\begin{theorem}[{\cite[Theorem 10.2]{pramod}}]
\label{achargluing}
Let $G$ be an affine algebraic group and let $Y$ be a scheme on which $G$ acts with finitely many orbits.  For each orbit $O$, let $\scrI_O \subset \scrO_Y$ denote the ideal sheaf corresponding to $\overline{O}$ with its reduced scheme structure.  Suppose that each orbit is endowed with an s-structure $A_O$, and that the following two conditions hold:
\begin{enumerate}
\item[(F1)] For each orbit $O$, $(i_O^* \scrI_O)\vert_O \in {^{A_O}\!}\scrC_G(O)_{\leq 0}$.
\item[(F2)] Each $\scrF \in {^{A_O}\!}\scrC_G(O)_{\leq w}$ admits an extension $\scrF_1 \in \scrC_G(\overline{O})$ whose restriction to any smaller orbit $O' \subset \overline{O}$ (via $i_{O'}^*$) is in ${^{A_{O'}}\!}\scrC_G(O')_{\leq w}$.
\end{enumerate}
Then there is a unique s-structure $\bA$ on $Y$ which induces the given s-structure $A_O$ on each orbit.  It is given by
$$\spA \scrC_{\leq w} = \{\scrF \mid i_O^* \scrF \in {^{A_O}\!}\scrC_G(O)_{\leq w} \text{ for each orbit }O\}$$
\end{theorem}

We have the following corollary: if $Y$ is covered by $G$-stable Zariski open subsets, a compatible family of s-structures on the charts induces an s-structure on $Y$.  (Thus, s-structures form a sheaf in a very coarse topology on $Y$.)

\begin{corollary}
\label{cor-local}
Let $G$ be an affine algebraic group and let $Y$ be a scheme on which $G$ acts with finitely many orbits.  Suppose $Y = U \cup V$ where $U$ and $V$ are $G$-stable open subsets of $Y$.  Let $\bA_U$ be an s-structure on $U$ and let $\bA_V$ be an s-structure on $\bA_V$.  Then if $\bA_U$ and $\bA_V$ induce the same s-structure $\bA_{U \cap V}$ on $U \cap V$, there is a unique s-structure on $Y$ that induces $\bA_U$ and $\bA_V$.
\end{corollary}

\begin{proof}
The s-structures $\bA_U$ and $\bA_V$, since they agree over $U \cap V$, determine a unique s-structure $A_O$ on each orbit $O$.  The $A_O$ clearly satisfy (F1), and they satisfy (F2) on $U$ and $V$ separately.  Let us show they satisfy (F2) on all of $Y$.  

Let $O$ be an orbit and let $\scrF$ belong to ${^{A_O}\!}\scrC_G(O)_{\leq w}$.  If $\overline{O}$ is contained in either $U$ or $V$, then we may find the required extension by appealing to (F2) in $U$ or $V$.  Otherwise, we have that $\overline{O} \cap U$ and $\overline{O} \cap V$ are nonempty, which implies that $O \subset U \cap V$.  Let $\scrF_1$ be an extension of $\scrF$ to ${^{\bA_U}\!}\scrC_G(U)_{\leq w}$ and let $\scrF_2$ be an extension of $\scrF$ to ${^{\bA_V}\!}\scrC_G(V)_{\leq w}$.  Let $\scrG$ be the kernel of the difference map
$$\scrF_1 \vert_{U \cap V} \oplus \scrF_2 \vert_{U \cap V} \to j_{O*} \scrF$$
where $j_O$ denotes the inclusion of $O$ into $U \cap V$.  Note that the restriction of $\scrG$ to $O$ is naturally isomorphic to $\scrF$ -- in fact to the diagonal copy of $\scrF$ in $\scrF_1 \vert_O \oplus \scrF_2 \vert_O \cong \scrF \oplus \scrF$.  Since $\scrG$ is a subsheaf of the coherent sheaf $\scrF_1 \vert_{U \cap V} \oplus \scrF_2 \vert_{U \cap V}$, it is coherent, and since $\scrF_1 \vert_{U \cap V}$ and $\scrF_2 \vert_{U \cap V}$ belong to the Serre subcategory ${^{\bA_{U \cap V}}\!}\scrC_G(U \cap V)_{\leq w}$, $\scrG$ belongs to ${^{\bA_{U \cap V}}\!}\scrC_G(U \cap V)_{\leq w}$ as well.

By the definition of induced s-structures $\scrG$ admits extensions $\scrG_1$ to ${^{\bA_U}\!}\scrC_G(U)_{\leq w}$ and $\scrG_2$ to $^{\bA_V}\!\scrC_G(V)_{\leq w}$.  We may glue $\scrG_1$ and $\scrG_2$ along the isomorphism $\scrG_1 \vert_{U \cap V} \cong \scrG \cong \scrG_2 \vert_{U \cap V}$ to obtain a sheaf $\scrG_3$ on $Y$.  Then for each orbit $O'$, $i_{O'}^* \scrG_3$ is isomorphic either to $i_{O'}^* \scrG_1$ or $i_{O'}^* \scrG_2$ depending on whether $O'$ belongs to $U$ or to $V$; in particular $i_{O'}^*$ belongs to ${^{A_{O'}}\!}\scrC_G(O')_{\leq w}$, as desired. 
\end{proof}

\subsection{Step and altitude}
\label{stepalt}

The \emph{altitude} is an important numerical invariant of an s-structure.  It is an integer associated to each pair $(O,\bbD)$, where $\bbD$ is an equivariant Serre-Grothendieck dualizing complex on $Y$, and $O \subset Y$ is a $G$-orbit.  We will define the altitude as the \emph{step} of a certain line bundle restricted to $O$.

\begin{definition}
\label{def-step}
Let $\scrC$ be an abelian category and let $\{\scrC_{\leq w}\}$ be a Serre filtration of $\scrC$.  An object $c \in \scrC$ is called \emph{pure of step $w$} with respect to the Serre filtration if it belongs to $\scrC_{\leq w}$ and if it contains no subobjects which belong to $\scrC_{\leq w-1}$.  In this case we will write $\text{step }c = w$. 
\end{definition}

\begin{remark}
A simple object of $\scrC$ is always pure of some step; in particular, if $O$ is a $G$-orbit then any line bundle in $\scrC_G(O)$ is pure of some step.  By (S6) and (S8), this defines a homomorphism $\text{step}:\Pic_G(O) \to \bbZ$.  If $O$ is a $G$-orbit in a more general $G$-scheme $Y$ equipped with an s-structure, define a homomorphism $\text{step}_O:\Pic_G(Y) \to \bbZ$ by $L \mapsto \step(i_O^* L)$.
\end{remark}

The dualizing complexes form a torsor for $\Pic_G(Y) \times \bbZ$ by tensoring with line bundles and shifting.  The altitude is invariant under shifts in general, and we find it convenient to further normalize so that it is a function on $\Pic_G(Y)$, by picking a distinguished dualizing complex $K_Y$.

When $Y$ is smooth, let $K_Y = \Omega^{\mathrm{dim}(Y)}[\mathrm{dim}(Y)]$ denote the canonical bundle of $Y$, placed in cohomological degree $-\mathrm{dim}(Y)$ and endowed with its natural $G$-equivariant structure.  It has the property that $\bR i_Z^!(K_Y) = K_Z$ for any smooth $G$-stable subvariety $Z$ of $Y$.  For a general $Y$ let $K_Y \in D^b_G(Y)$ denote, if it exists, the equivariant Serre-Grothendieck complex with the property that $\bR i_Z^! K_Y \cong K_Z$ when $Z$ is a smooth $G$-stable subscheme of $Y$.  Such a $K_Y$ can be found, for instance, if $Y$ may be covered by $G$-stable affine open charts.  (When $Y$ is a toric variety, $K_Y$ is given explicitly in \cite{ishida}.)

\begin{warning}
Even for smooth $Y$, the sheaf $K_Y$ is not an invariant  of the quotient stack $[Y/G]$.  That is, 
if $[Y_1/G_1]$ and $[Y_2/G_2]$ are equivalent as algebraic stacks, $K_{Y_1}$ and $K_{Y_2}$ may not coincide under the identification $\scrC_{G_1}(Y_1) \cong \scrC_{G_2}(Y_2)$.
\end{warning}

\begin{definition}
\label{alt}
Let $G$ be an affine algebraic group and let $Y$ be a scheme on which $G$ acts with finitely many orbits.  Let $K_Y$ be as above, let $L$ be an equivariant line bundle on $Y$, and let $O$ be a $G$-orbit of $Y$.  Let us define the \emph{altitude} $\alt(O,L)$ of $L$ at $O$ to be $\text{step}_O (K_O \otimes L) = \text{step}_O(K_O) + \text{step}_O(L)$.
\end{definition}

It is easy to see that the altitude of $Y$ in the sense of \cite[Definition 6.3]{pramod} is equal to $\alt(O;L)$ when $O$ is the open orbit and we take the dualizing complex $\omega_Y$ to be $K_Y \otimes L$.  Similarly, the altitude of a more general orbit closure $\overline{O'}$ with its induced s-structure, and with $\omega_{\overline{O'}} = \bR i^! (\omega_Y)$, is $\alt(O';L)$ in light of the general formula $\bR i^!(K_Y \otimes L) \cong K_{\overline O} \otimes i^* L$ which is valid when $L$ is a line bundle.

\section{$\mathrm{s}$-structures on toric varieties}

In this section we will apply Achar's gluing theorem for s-structures (\ref{achargluing}) to toric varieties.  Our main result is theorem \ref{thm-stor}, which provides a bijection between s-structures on toric varieties and combinatorial data on the fan associated to $Y$.

\subsection{Toric varieties}
\label{sec-torvar}

To fix our notation let us review the connection between toric varieties and fans.
Let $T \cong \bbG_m^n$ be an $n$-dimensional algebraic torus, split over the base $\bbF$.  Let $X_*(T) = \Hom(\bbG_m, T)$ and $X^*(T) = \Hom(T,\bbG_m)$ denote the dual lattices of cocharacters and characters, respectively.  
We will mostly use additive notation for the group law in these lattices, and when multiplicative notation is more convenient (for instance when regarding a character as an actual function on $T$ or some compactification of $T$) we will use the notation $e^\lambda$.  Write $\langle, \rangle$ for the natural pairing $X^*(T) \otimes X_*(T) \to \bbZ$.

A \emph{$T$-toric variety} $Y$ is a reduced, normal, $n$-dimensional $T$-variety, with the property that $T$ acts simply transitively on a dense open subset of $Y$.  Let us pick a base-point $y_0$ in this open orbit.  For a $T$-orbit $O$ let $\st(O) \subset Y$ be the union of orbits $O'$ with $O \subset \overline{O'}$; $\st(O)$ is a $T$-stable open affine subvariety of $Y$.  To each $O$ we may associate the following cone in $X_*(T)$:
$$C(O) = \{ \gamma \in X_*(T) \mid \lim_{u \to 0} e^\gamma(u) \cdot y_0 \in \st(O)\}$$
$C(O)$ is a strongly convex rational polyhedral cone in $X_*(T)$.  That is, it is the intersection of finitely many rational half-spaces in $X_*(T)$ and it contains no line through the origin.  It does not depend on the base point.

The cones $C(O)$ form a \emph{fan} in $X_*(T)$, which we denote by $\Sigma(Y)$.  That is, the face of any cone of the form $C(O)$ is of the form $C(O')$, and the intersection of two cones $C(O_1) \cap C(O_2)$ is a face of both $C(O_1)$ and $C(O_2)$.

$C(O)$ has a dual cone $C(O)^\vee \subset X^*(T)$, which consists of those $\chi$ such that $\langle\chi,\gamma\rangle \geq 0$ for all $\gamma \in C(O)$.  $C(O)^\vee$ is a convex rational polyhedral cone, but it is not always strongly convex (for instance when $O$ is the open orbit $C(O)^\vee$ is all of  $X^*(T)$).  An equivalent description of $C(O)^\vee$ is
$$\{\chi \mid e^\chi \text{ extends to a regular function on }\st(O)\}$$
where we regard characters of $T$ as rational functions on $Y$ using the base-point $y_0$.  Thus we may identify the ring of functions on $\st(O)$ with $\bbF[C(O)^\vee]$, the semigroup ring of $C(O)^\vee$.  

Given a cone $C$ in $\Lambda \cong \bbZ^n$, let $\dim(C)$ denote the rank of the sublattice spanned by $C$, and set $\cod(C) = n - \dim(C)$.  We have $\dim(O) = \cod(C(O))$.

Finally, the following notation will be useful.  If $V$ is a representation of $T$, let us write $\supp_T(V) \subset X^*(T)$ for the set of weights of $T$ on $V$.  Note that $\supp_T 0 = \emptyset$.  If $\scrF$ is a $T$-equivariant quasicoherent sheaf on an affine $T$-variety, define $\supp_T(\scrF) := \supp_T H^0(Y;\scrF)$.    
If $K$ is a complex of $T$-modules or $T$-equivariant sheaves, set $\supp_T K = \bigcup_i \supp_T h^i(K)$.

\subsection{s-structures on torus orbits}

Let $T$ be a torus, and suppose that $T$ acts transitively on a variety $O$.  Let $T_O \subset T$ denote the isotropy subgroup for the $T$-action.  An s-structure on $O$ determines a homomorphism $\step:\Pic_T(O) \to \bbZ$ (definition \ref{def-step}), and since $\Pic_T(O) \cong X^*(T_O)$ we can identify this homomorphism with an element of $X_*(T_O) \subset X_*(T)$.
Conversely, for each $A \in X_*(T_O)$, define a Serre filtration of $\scrC_T(O)$
by
$${^A\!}\scrC_{\leq w} = \{\scrF \mid \langle \chi,A \rangle \leq w \text{ for all }\chi \in \supp_T(\scrF)\}$$

\begin{proposition}
\label{Torbit}
The Serre filtration associated to each $A \in X_*(T_O)$ is an s-structure, and the assignments $X_*(T_O) \leftrightarrow \{\text{s-structures on $O$}\}$ described above are inverse bijections.
\end{proposition}

\begin{proof}
To show that $\{{^A\!}\scrC_{\leq w}\}$ is an s-structure we have to verify conditions (S4)-(S8) of definition \ref{def-s-struct}.  Each of these verifications is trivial: (S4) and (S5) follow from the fact that $\scrC_T(O)$ is semisimple and each simple object belongs to ${^A\!}\scrC_{\leq w} \cap {^A\!}\scrC_{\geq w}$ for some $w$.  (S6) and (S8) follow from the fact that $\langle-, A\rangle$ is additive, and (S7) follows from the fact that the subcategories $\scrC_{\geq w} \subset \scrC_T(O)$ are Serre. 

Let $L$ be an equivariant line bundle on $O$, with weight $\chi \in X^*(T_O)$.  Then $L \in {^A\!}\scrC_{\leq {\langle \chi,A\rangle}}$ and 
$L \notin {^A\!}\scrC_{\leq {\langle\chi,A\rangle} -1}$ by definition.  Thus, $\step(L) = \langle \chi,A \rangle$. 
 This shows that the compositions $X_*(T_O) \to \{\text{s-structures on }O\} \to X_*(T_O)$ and 
 $\{\text{s-structures}\} \to X_*(T_O) \to \{\text{s-structures}\}$ are the identity maps.  This completes the proof. 
\end{proof}

\begin{remark}
It is true for a general group $G$ that an s-structure on a $G$-orbit is determined by its step homomorphism $\Pic_G(O) \to \bbZ$.  E.g. when the isotropy subgroup $H$ of $G$ is semisimple there are no nontrivial s-structures.  It can be shown from (S4) that a homomorphism $\Pic_G(O) \to \bbZ$ determines an s-structure if and only if it carries $\supp_T \scrO_H$ to nonnegative numbers, where $\scrO_H$ is the ring of regular functions on $H$ and $T$ acts on $\scrO_H$ by conjugation.  This is an observation of Achar's.
\end{remark}

\subsection{s-structures on fans}

We have seen that s-structures on $T$-orbits are given by cocharacters of $T$ which lie in the isotropy subgroup.  By theorem \ref{achargluing} an s-structure on a toric variety $Y$ will be given by attaching a cocharacter of $T$ to each orbit, satisfying certain conditions.  The following definition shows how these conditions can be expressed in terms of the fan associated to $Y$.

\begin{definition}
\label{sfan}
Let $\Lambda$ be a lattice and let $\Sigma$ be a fan in $\Lambda$.  An s-structure on $\Sigma$ is an assignment that associates to each cone $C\subset \Lambda$ of $\Sigma$ a distinguished element $A_C \in X_*(T)$, subject to the following conditions:
\begin{enumerate}
\item $A_C \in -C$ for all cones $C$.
\item Whenever $A_C = 0$ and $C' \subset C$ is a face, $A_{C'}= 0$ as well.
\end{enumerate}  
If $Y$ is a toric variety with fan $\Sigma$ and $\bA = \{A_C\}_{C \in \Sigma}$ is an s-structure on $\Sigma$, define a Serre filtration on $\scrC_T(Y)$ by
$$\spA\scrC_{\leq w} = \{\scrF \mid \langle \chi,A_C\rangle \leq w \text{ for all }\chi \in \supp_T i_C^* \scrF\}$$
Here $i_C: O \hookrightarrow Y$ is the inclusion of the reduced $T$-orbit corresponding to $C$, and $T_C \subset T$ is the isotropy subgroup of $O$.
\end{definition}

It is convenient to discuss the s-truncation functors $\sigma_{\leq w}$ associated to an s-structure on a fan before proving (theorem \ref{thm-stor}) that the Serre filtrations $\spA \scrC_{\leq w}$ actually form an s-structure on $\scrC_T(Y)$.  The functors $\sigma_{\leq w}$ behave somewhat differently for $w \geq 0$ than for $w < 0$: for negative $w$ the sheaf $\sigma_{\leq w}$ is always set-theoretically supported on a certain closed subset of $Y$ associated to the s-structure $\bA$.

\begin{definition}
\label{def-strunctor}
Let $Y = \Spec(\bbF[C^\vee])$ be an affine toric variety with torus $T$, and let $\Sigma$ be the fan in $X_*(T)$ associated to $Y$.  (So $\Sigma$ is the set of faces $D \subset C$).  Let $\bA = \{A_D\}_{D \in \Sigma}$ be an s-structure on $\Sigma$ and let $\spA \scrC_{\leq w}$ be the associated Serre filtration.
\begin{enumerate}
\item Let $Z \subset Y$ be the closed union of orbits in $Y$ corresponding to cones $C$ with $A_C \neq 0$.  Let $\hat{i}_Z^!:\scrC_T(Y) \to \scrC_T(Y)$ be the functor which takes a sheaf $\scrF$ to the maximal subsheaf set-theoretically supported on $Z$.
\item Define a functor $\sigma'_{\leq w}:\scrC_T(Y) \to \scrC_T(Y)$ by setting $\sigma_{\leq w} \scrF$ to be the maximal subsheaf of $\scrF$ with 
$$\langle \chi,A_D\rangle \leq w \text{ whenever } \chi \in \supp_T \scrF$$
\item Define $\sigma_{\leq w} := \sigma'_{\leq w}$ when $w \geq 0$, and $\sigma_{\leq w} = \sigma'_{\leq w} \circ \hat{i}_Z^!$ when $w < 0$.
\end{enumerate}
\end{definition}

Note that $H^0(\sigma'_{\leq w} \scrF)$ is the vector subspace of  $H^0(\scrF)$ spanned by the weight spaces $H^0(\sigma'_{\leq w} \scrF)_\chi$ with $\langle \chi,A_D\rangle \leq w$ for all $D$.  The condition $A_D \in -D$ implies that this vector subspace is actually a sub-$\bbF[C^\vee]$-module.  Using this fact it is easy to verify that $\sigma_{\leq w} \scrF$ always belongs to $\scrC_T(Y)_{\leq w}$.

\begin{theorem}
\label{thm-stor}
Let $Y$ be a toric variety with torus $T$, and let $\Sigma(Y)$ be the fan in $X_*(T)$ associated to $Y$.  Let $\bA = \{A_C\}_{C \in \Sigma(Y)}$ be an s-structure on $\Sigma(Y)$, and let $\{\spA \scrC_{\leq w}\}_{w \in \bbZ}$ be the Serre filtration on $\scrC_T(Y)$ of definition \ref{sfan}.
This Serre filtration is an s-structure on $Y$, and the assignment $\bA \mapsto \{\spA\scrC_{\leq w}\}$ is a bijection between s-structures on $\Sigma(Y)$ and s-structures on $Y$.
\end{theorem}

\begin{proof}
By proposition \ref{Torbit}, the $A_C$ define an s-structure on each orbit, and by theorem \ref{achargluing}, it suffices to show that the collections $\{A_C \in X_*(T_C)\}$ that satisfy conditions (F1) and (F2) are precisely those with $A_C \in -C$ for each $C$.  

Let us first show that condition (F1) of theorem \ref{achargluing} is equivalent to condition (1) of definition \ref{sfan}.  That is, let us show that $i_C^* \scrI_C \vert_C \in \scrC_{\leq 0}$ if and only if $A_C \in -C$, where $\scrI_C$ is the ideal sheaf of the orbit closure $\overline{O}$ corresponding to $C$.   We have $i_C^* \scrI_C \in \scrC_{\leq 0}$ if and only if  $\langle \theta, A_C\rangle \leq 0$ for each $\theta \in \supp_T(i_C^* \scrI_C)$, and we have to show that 
$$- C = \{A_C \in X_*(T) \mid \langle \theta,A_C \rangle \leq 0 \text{ for all }\theta \in \supp_T(i_C^* \scrI_C)\}$$
Since $-C = (-C)^{\vee \vee}$, it suffices to show that $\supp_T(i_C^* \scrI_C)$ spans $C^\vee$ in the sense that $\bbZ_{\geq 0} \cdot \supp_T(i_C^* \scrI_C) = C^\vee$.  

By restricting to a suitable $T$-stable open subset, we may assume that $Y = \Spec(\bbF[C^\vee])$ is affine and that $O$ is already closed in $Y$, and since such a toric variety is of the form $O \times Y'$ we may further reduce to the case where $O$ is a point.  Then $\scrI_C$ corresponds to the ideal $I_C \subset \bbF[C^\vee]$ of functions which vanish on $O$, and $i^* \scrI_C$ corresponds to the vector space $I_C/I_C^2$.  The ideal $I_C$ has a basis of characters given by nonzero elements of $C^\vee$, so its square $I_C^2$ has a basis of characters belonging to $C^\vee$ that can be written as a sum of two or more nonzero elements of $C^\vee$.  Thus, the characters in $\supp_T(i_C^* \scrI_C)$ are the indecomposable elements of $C^\vee$.  Since $O$ is a point, $C^\vee$ is strongly convex, and these indecomposable elements span $C^\vee$ as required.

Now assume that condition (F1) is satisfied, or equivalently (as we have just shown) that each $A_C$ lies in $-C$.  Then let us show that (F2) is equivalent to condition (2) of definition \ref{sfan}.  
First let us show that condition (F2) implies condition (2).  Let $C' \subset C$ be a pair of cones and $O' \supset O$ the corresponding pair of orbits.  Suppose that $A_{C'} \neq 0$.  Then there exists a nonzero object $\scrF \in \scrC_T(O')_{\leq -1}$.  If (F2) holds then $\scrF$ admits an extension $\scrF_1 \in \scrC_T(\overline{O}')_{\leq -1}$.  This implies by definition that $i_O^* \scrF_1 \in \scrC_T(O)_{\leq -1}$.  By Nakayama's lemma, $i_O^* \scrF_1$ must be nonzero.  This implies $A_C$ is not zero -- otherwise there would be no nonzero objects in $\scrC_T(O)_{\leq -1}$.

Now let us show the converse.  By corollary \ref{cor-local}, to show that $\bA$ is an s-structure we may assume that $Y$ is affine, say $Y = \Spec(\bbF[C^\vee])$.  Let $O$ be an orbit and let $D$ be the cone corresponding to $O$.  We have to show that any $T$-equivariant coherent sheaf $\scrF$ on $O$ belonging to ${^{A_D}\!}\scrC_T(O)_{\leq w}$ admits an extension to $\spA \scrC_T(\overline{O})_{\leq w}$.
We may assume that $\scrF$ is an equivariant line bundle (and in particular $\scrF \neq 0$, so that either $w \geq 0$ or $A_D \neq 0$).  Then $\scrF$ extends to a line bundle on $\overline{O}$: if $\xi \in X^*(T)$ is an element of $\supp_T \scrF$, and $L(\xi)$ is the line bundle on $Y$ whose $X^*(T)$-graded $\bbF[C^\vee]$-module of sections is freely generated by an element of weight $\xi$, then $L(\xi) \vert_O \cong \scrF$.  Set $\scrF_1 = \sigma'_{\leq w} L(\xi)$.  Since $\scrF \in {^{A_D}\!}\scrC_T(O)_{\leq w}$ we have $\langle \xi,A_D\rangle \leq w$, and it follows that $\scrF_1 \vert_O \cong \scrF$.
\end{proof}

\subsection{Altitude}
\label{sec-alttor}

Finally in this section, we compute the altitude function (see definition \ref{alt}) $\alt(O;-):\Pic_T(Y) \to \bbZ$ of an s-structure on a toric variety.  If $O$ is associated to a cone $C$ in the fan $\Sigma(Y)$, let us write $\alt(C;-) = \alt(O;-)$ and $\step_C = \step_O$.

First let us recall how to identify $\Pic_T(Y)$ with the group of piecewise-linear functions on the fan associated to $Y$.

Let $\uchi$ be a piecewise linear function on the fan $\Sigma$ associated to $Y$ -- i.e. a $\bbZ$-valued function on the union of cones $C \subset X_*(T)$ which is additive on each cone $C$.  
Note that $\uchi$ determines (indeed, is equivalent to) the data of, for each $C$, an element $\uchi_C \in X^*(T_C) \cong \Hom(X_*(T_C),\bbZ)$, subject to the condition that 
$$\uchi_C \vert_{X_*(T_{C'})} = \uchi_{C'}$$
whenever $C'$ is a face of $C$.  Here $T_C \subset T$ denotes the isotropy subgroup of the orbit corresponding to $C$.

Then up to isomorphism there is a unique $T$-equivariant line bundle $L(\uchi) \in \Pic_T(Y)$ such that, for each cone $C$, the character of $T_C$ on the restriction of $L(\uchi)$ to the orbit associated to $C$ is $\uchi_C$.  Moreover, every equivariant line bundle arises in this way.

\begin{proposition}
\label{alttor}
Let $T$ be a torus and let $Y$ be a $T$-toric variety with fan $\Sigma$ in $X_*(T)$.  Let $\bA = \{A_C \in -C\}_{C \in \Sigma}$ be an s-structure on $\Sigma$, and let $\uchi$ be a piecewise linear function on the  fan $\Sigma$.
\begin {enumerate}
\item For each cone $C$, we have $K_C \cong \scrO[\cod(C)]$.
\item The altitude of $L(\uchi)$ on the orbit associated to $C$ is given by
$$\alt(C,L) = \uchi(-A_C)$$
\end{enumerate}
\end{proposition}

\begin{proof}
The first assertion is trivial.  To prove the second assertion, recall that the formula of definition \ref{alt} has $\alt(C;L) = \step_C(K_C) +\step_C(L)$.  The weight of $T_C$ on $K_C$ is zero, so by proposition \ref{Torbit}, the step of $L$ at the orbit associated to $C$ is $\langle \uchi_C,-A_C\rangle$.
\end{proof}

\section{The staggered t-structure}

Let $Y$ be a scheme on which $G$ acts with finitely many orbits.  Given an s-structure $\bA$ on $Y$ and a function $p:\{G\text{-orbits on }Y\} \to \bbZ$, we may define a pair of subcategories
$$
\begin{array}{c}
{^{\bA,p}\!}D^{\leq 0} \subset D^b_G(Y) \\
{^{\bA,p}\!}D^{\geq 0} \subset D^b_G(Y)
\end{array}
$$
by putting $\scrF \in \spAp D^{\leq 0}$ whenever $h^k(Li_O^*\scrF) \in \spA\scrC_{\leq p(O)-k}$, and defining $\spAp D^{\geq 0}$ to be the right orthogonal to $\spAp D^{\leq -1} := \spAp D^{\leq 0}[1]$.  One of the main results of \cite{pramod} is that under certain ``perversity'' conditions on $(\bA,p)$, the pair $(\spAp D^{\leq 0}, \spAp D^{\geq 0})$ is a t-structure in the sense of \cite{bbd}.

The perversity conditions are subject to an auxiliary choice -- that of a dualizing sheaf $\bbD$.  Let us set $\bbD = K_Y \otimes L$ for an equivariant line bundle $L$, as in section \ref{stepalt}.  

\begin{definition}
\label{defperv}
Let $G$ be an algebraic group and let $Y$ be a scheme on which $G$ acts with finitely many orbits.  Let $\bA$ be an s-structure on $Y$, and let $\bbD$ be a dualizing sheaf of the form $K_Y \otimes L$.  We will say that a $\bbZ$-valued function $p$ on the $G$-orbits of $Y$ is a \emph{perversity with respect to $L$} if for each pair of orbits $O,O'$ with $O' \subset \overline{O}$, we have
$$0 \leq p(O') - p(O) \leq (-\dim(O) + \alt(O;L)) - (-\dim(O')+ \alt(O';L))$$
In \cite{pramod}, the quantity $(-\dim(O)+\alt(O;L))$ is called the ``staggered codimension'' of $O$ (with respect to the dualizing complex $K_Y \otimes L$).  The lower (resp. upper) bound on $p(O') - p(O)$ is called the ``monotonicity'' (resp. ``comonotonicity'') condition.  We will say that $p$ is a \emph{perversity function} if it is a perversity function with respect to some line bundle $L$.
\end{definition}

\begin{remarks}
\label{dualperv}
Let $p$ be a perversity function with respect to $L \in \Pic_G(Y)$.  It is proved in \cite{pramod} that $\spAp D^{\geq 0}$ has a more concrete description.  Let $\overline{p}$ denote the function 
$$\overline{p}(O) = -\dim(O) + \alt(O;L) - p(O)$$
Note that $\overline{p}$ is also a perversity with respect to $L$.  Then
$$\spAp D^{\geq 0} = \{\scrF \mid \bbD \scrF := \bR \underline{\Hom} (\scrF,K_Y \otimes L)
\text{ is in }{^{\bA,\overline{p}}\!}D^{\leq 0}
\}$$
\end{remarks}

Now we can state the main result of \cite{pramod} (in the special case when $G$ acts with finitely many orbits).

\begin{theorem}[{\cite[Theorem 7.4]{pramod}}]
\label{thm74}
Let $G$ be an algebraic group and let $Y$ be a scheme on which $G$ acts with finitely many orbits.  Let $\bA$ be an s-structure on $Y$ and let $p$ be a perversity function.  Then $(\spAp D^{\leq 0},\spAp D^{\geq 0})$ is a bounded, nondegenerate t-structure on $D^b_G(Y)$.  Let us call it the \emph{$(\bA,p)$-staggered t-structure}.
\end{theorem}

\subsection{Staggered t-structures on toric varieties}

Now we will apply theorem \ref{thm74} to toric varieties.  First we express definition \ref{defperv} in terms 
of fans.

\begin{definition}
\label{pfan}
Let $\Lambda \cong \bbZ^n$ be a lattice and let $\Sigma$ be a fan in $\Lambda$.  Let $\bA = \{A_C \in -C\}_{C \in \Sigma}$ be an s-structure on $\Sigma$, and let $\uchi$ be a $\bbZ$-valued piecewise-linear function on the antipodal fan to $\Sigma$.  A $\bbZ$-valued function $p: C \mapsto p(C)$ on the set of cones in $\Sigma$ is called a \emph{perversity function} with respect to $\uchi$ if, whenever $C$ is a face of $C'$, 
$$0 \leq p(C') - p(C) \leq \mathrm{dim}(C') + \uchi(-A_{C'}) -\mathrm{dim}(C) - \uchi(-A_C)$$
Note that to verify $p$ is a perversity with respect to $\uchi$, it is enough to check the inequalities
$$0 \leq p(C') - p(C) \leq 1 + \uchi(-A_{C'}) - \uchi(-A_C)$$
when $C$ is a codimension one face of $C'$.
We will say that $p$ is a perversity function if it a perversity function with respect to some $\uchi$.
\end{definition}

\begin{remark}
Note that the existence of a perversity function implies some weak convexity conditions on the parameters $A_C$ defining an s-structure.  Namely, for $\bA$ to admit a perversity function there must exist a piecewise-linear function $\uchi$ satisfying
$$\uchi(-A_C') - \uchi(-A_C) +1 \geq 0$$
for all cones $C'$ and $C$ with $C$ a codimension one face of $C'$.
\end{remark}

If $\Lambda = X_*(T)$ and $\Sigma = \Sigma(Y)$ for some torus $T$ and some $T$-toric variety $Y$, it is clear that a perversity function on $\Sigma$, perverse with respect to $\uchi$, induces a perversity function on $Y$ perverse with respect to $L(\uchi)$.  (See section \ref{sec-alttor}.)

A straightforward translation of theorem \ref{thm74} into the language of fans reads

\begin{theorem}
\label{thmpervtor}

Let $T \cong \bbG_m^n$ be an $n$-dimensional torus, and let $Y$ be a $T$-toric variety with associated fan $\Sigma$ in $X_*(T)$.  Let $\bA$ be an s-structure and let $p$ be a perversity function.  For each $C \in \Sigma$ let $i_C$ denote the inclusion of the orbit associated to $C$ into $Y$.  Define full subcategories
$\spAp D^{\leq 0}$ and $\spAp D^{\geq 0}$ of $D^b_T(Y)$, by
$$\begin{array}{ccc}
\spAp D^{\leq 0} & =  & \{ \scrF \mid \forall C\, \langle \xi, A_C\rangle \leq p(C) - k \text{ whenever }\xi \in \supp_T h^k(\bL i_C^* \scrF)\} \\
\spAp D^{\geq 0} & = & \{ \scrF \mid \forall C \, \langle \xi, A_C \rangle \geq p(C) - k \text{ whenever } \xi \in \supp_T h^k(\bR i_C^! \scrF)\}
\end{array} 
$$
Then $(\spAp D^{\leq 0}, \spAp D^{\geq 0})$ is a bounded, nondegenerate t-structure on $D^b_T(Y)$.
\end{theorem}

We will denote the heart of this t-structure by $\spAp \scrM_T(Y)$, and call it the category of \emph{$(\bA,p)$-staggered sheaves} on $Y$.

\begin{proof}

Suppose that $p$ is a perversity with respect to $\uchi$, and let $\overline{p}$ be the dual perversity given by 
$$\overline{p}:C \mapsto -\cod(C)+ \uchi(-A_C) - p(C)$$
After theorem \ref{thm74} and remark \ref{dualperv}, we only have to show that the subcategory $\spAp D^{\geq 0}$ defined in the statement of the theorem coincides with
$$\{
\scrF \mid \bbD_Y \scrF \text{ is in }{^{\bA,\overline{p}}\!}D^{\leq 0}
\}
$$
Here $\bbD_Y$ denotes the dualizing functor $\bR\underline{\Hom}(-,K_Y \otimes L(\uchi))$.  
But we have
$$\begin{array}{rl}
\supp_T \bigg( h^k \bL i_C^* \bbD \scrF\bigg) & = \supp_T \bigg( h^k \bR\underline{\Hom}(\bR i_C^! \scrF,L(\uchi_C)[\cod(C)]) \bigg) \\
& = \supp_T \bigg(h^{k + \cod(C)}\bR\underline{\Hom}(\bR i_C^! \scrF,L(\uchi_C))\bigg) \\
& = \supp_T\bigg(\Hom(h^{-\cod(C)-k} \bR i_C^! \scrF,L(\uchi_C))\bigg) \\
& = \uchi_C - \supp_T \bigg(h^{-\cod(C)-k} \bR i_C^! \scrF\bigg)
\end{array}
$$

Now, by definition, $\bbD \scrF \in {^{\bA,\overline{p}}\!}D^{\leq 0}$ if and only if 
$$\langle \xi,A_C \rangle \leq \overline{p}(C) - k \text{ whenever }\xi \in \supp_T h^k \bL i_C^* \bbD\scrF$$
which, after the equations above, can be rewritten as
$$\langle \theta,A_C\rangle \geq p(C) - k \text{ whenever } \theta \in \supp_T h^k \bR i_C^! \scrF$$
proving the theorem.
\end{proof}

\begin{remark}

I do not know the most general hypotheses on the function $p$ that guarantees that $(\spAp D^{\leq 0}, \spAp D^{\geq 0})$ is a t-structure, but the condition ``$p$ is a perversity function in the sense of definition \ref{pfan}'' is only sufficient, not necessary.  Tensoring the subcategories associated to $(\bA,p)$ with an equivariant line bundle gives the subcategories associated to $(\bA,C \mapsto p(C) + \uchi(-A_C))$ for a piecewise linear function $\uchi: X_*(T) \to \bbZ$.  It is possible for $p$ to be a perversity function but for $C \mapsto p(C) + \uchi(-A_C)$ not to be one.
\end{remark}

\begin{corollary}
\label{cor-dual}
Let $T \cong \bbG_m^n$, $Y$,  and $\bA$ be as in theorem \ref{thmpervtor}.  Let $L(\uchi)$ be an equivariant line bundle on $Y$.  Suppose that the quantity
$\dim(C) + \uchi(-A_C)$
is even for all $C$.  Suppose furthermore that, for all cones $C,C'$ with $C$ a codimension one face of $C'$, we have
$$\uchi(-A_{C'}) - \uchi(-A_C) + 1 \geq 0$$
Then the function
$$p(C) := \frac{1}{2}(\dim(C) + \uchi(-A_C))$$
is a perversity with respect to $\uchi$, and the $(\bA,p)$-staggered t-structure is self-dual with respect to the dualizing functor defined by
$$\bbD \scrF := \bR \underline{\Hom}(-,K_Y \otimes L(\uchi)[n])$$

\end{corollary}

\begin{proof}
It is clear that the inequality $\uchi(-A_{C'}) - \uchi(-A_C) + 1 \geq 0$ implies the inequalities
$$0 \leq \frac{1}{2}(\dim(C') + \uchi(-A_{C'})) - \frac{1}{2}(\dim(C) + \uchi(-A_C)) \leq 1 + \uchi(-A_{C'} )- \uchi(-A_C)$$
and so $p(C)$ is a perversity with respect to $\uchi$.

As in the proof of theorem \ref{thmpervtor}, and taking account of the shift-by-$n$ in $\bbD$, we find that for any $\scrF \in D^b_T(Y)$, a character $\xi$ is in $\supp_T(h^{k+n} \bL i_C^* \bbD \scrF)$ if and only if $\uchi_C - \xi$ is in $\supp_T(h^{-\cod(C) - k} \bR i_C^! \scrF)$.  From this, we may see that each of the following statements is equivalent to the one preceeding it:
\begin{enumerate}
\item $\bbD \scrF \in \spAp D^{\leq 0}$.
\item $\langle \xi,A_C \rangle \leq p(C) - k - n$ for all $C$ and for all $\xi \in \supp_T(h^{k+n} \bL i_C^* \bbD \scrF)$.
\item $\langle \uchi_C - \theta,A_C \rangle \leq p(C) - k - n$ for all $C$ and for all $\theta \in \supp_T(h^{-\cod(C) - k} \bR i_C^! \scrF)$.
\item $\langle \theta,A_C \rangle \geq p(C) - \ell$ for all $C$ and for all $\theta \in \supp_T(h^{\ell} \bR i_C^! \scrF)$
\item $\bR i_C^! \scrF \in D^{\geq 0}$ for all $C$
\end{enumerate}
It follows that $\scrF \in \spAp D^{\geq 0}$ if and only if $\bbD \scrF \in \spAp D^{\leq 0}$.  Since $\bbD \circ \bbD \scrF \cong \scrF$, it also follows that $\scrF \in \spAp D^{\leq 0}$ if and only if $\bbD \scrF \in \spAp D^{\geq 0}$.
\end{proof}

\emph{Acknowledgements:}  I would like to thank Pramod Achar for teaching me his theory of staggered sheaves, encouraging me to work out this example, and making many helpful comments on an early draft of this paper.

\end{document}